\documentclass[12pt,a4paper]{article}
\usepackage{amsthm}
\usepackage{amsmath}
\usepackage{epsfig}
\usepackage{color} 
\usepackage{amssymb}

\newtheorem{theorem}{Theorem}[section]
\newtheorem{proposition}[theorem]{Proposition}
\newtheorem{lemma}[theorem]{Lemma}
\newtheorem{corollary}[theorem]{Corollary}

\newtheorem{remark}[theorem]{Remark}
\newtheorem*{theorem1*}{Theorem 1}
\newtheorem*{theorem2*}{Theorem 2}
\newtheorem*{conjecture1*}{Conjecture 1}
\newtheorem*{conjecture2*}{Conjecture 2}

\begin{document}

\title{\large{\textbf{ODD DECOMPOSITIONS OF EULERIAN GRAPHS}}}

\author{
Edita M\'a\v cajov\' a and Martin \v{S}koviera\\[3mm]
Department of Computer Science\\
Faculty of Mathematics, Physics and Informatics\\
Comenius University\\
842 48 Bratislava, Slovakia\\[2mm]
{\small\tt macajova@dcs.fmph.uniba.sk}\\[-1mm]
{\small\tt skoviera@dcs.fmph.uniba.sk}}

\date{\today}

\maketitle

\begin{abstract}
We prove that an eulerian graph $G$ admits a decomposition into
$k$ closed trails of odd length if and only if and it contains
at least $k$ pairwise edge-disjoint odd circuits and $k\equiv
|E(G)|\pmod{2}$. We conjecture that a connected $2d$-regular
graph of odd order with $d\ge 1$ admits a decomposition into
$d$ odd closed trails sharing a common vertex and verify the
conjecture for $d\le 3$. The case $d=3$ is crucial for
determining the flow number of a signed eulerian graph which is
treated in a separate paper (arXiv:1408.1703v2). The proof of
our conjecture for $d=3$ is surprisingly difficult and calls
for the use of signed graphs as a convenient technical tool.

\end{abstract}

\section{Introduction}
Eulerian graphs constitute a fundamental class of graphs
extensively studied throughout the entire history of graph
theory. Among the many problems concerning eulerian graphs,
those related to decomposition into various types of subgraphs
belong to most typical in this area \cite{Fleischner1,
Fleischner2}. In fact, one of the main pillars of the eulerian
graph theory is the classical result of Veblen \cite{Veblen}
that the existence of a circuit decomposition in a connected
graph is equivalent for the graph to be eulerian.
Decompositions of eulerian graphs into circuits or trails of
restricted lengths, although natural to require, are generally
difficult to find and the known results are scarce.  In most
cases, the graphs known to have such decompositions are
complete multipartite graphs, thus having a very explicit
structure \cite{BC,BS,CH, HW, HK}.

In contrast, our paper focuses on the existence of closed trail
decompositions in general eulerian graphs, with the only
restriction that each member of the decomposition have an odd
number of edges. It is easy to see that a connected graph $G$
admitting a decomposition into $k$ odd closed trails must be
eulerian with $|E(G)|\equiv k\pmod{2}$ and has to contain at
least $k$ edge-disjoint odd circuits, one for each closed
trail. We show that this necessary condition is also sufficient
(Theorem~\ref{thm:k-odd}). In particular, every connected
$2d$-regular graph of odd order, with $d\ge 1$, admits a
decomposition into $d$ odd closed trails
(Corollary~\ref{cor:2k-reg}).

Our primary interest lies, however, in odd closed trail
decompositions where all members are required to share a common
vertex; we call such decompositions \textit{rooted}.
Equivalently, we ask whether an eulerian graph contains, for a
given positive integer $k$, an eulerian trail $T$ and a vertex
$v$ which divides $T$ into $k$ closed trails of odd length
based at~$v$. We show that for $k=2$ the answer is positive if
and only if the graph is non-bipartite and has an even number
of edges (Theorem~\ref{thm:2bly-odd}). As a consequence, the
following is true.

\begin{theorem1*}\label{thm:main1}
Every connected $4$-regular graph of odd order has a rooted
decomposition into two odd closed trails.
\end{theorem1*}

A substantial part of this paper is devoted to the problem of
decomposing an eulerian graph into three odd closed trails with
common origin. Our motivation comes from the area of
nowhere-zero flows on signed graphs. In \cite[Main
Theorem~(c)]{euler} we show that a signed eulerian graph admits
a nowhere-zero integer $3$-flow if and only if it has a rooted
decomposition into three closed trails with an odd number of
negative edges each (see also \cite[Theorem~2.4]{euler-short}).
After a series of natural reductions the proof amounts to
proving the following theorem which is the main result of the
present paper.

\begin{theorem2*}\label{thm:2}
Every connected $6$-regular graph of odd order has a rooted
decomposition into three odd closed trails.
\end{theorem2*}

The comparison of Theorem~\ref{thm:main1} and Theorem~2
suggests that the following might be true.

\begin{conjecture1*}\label{conj:d-flower}
Every connected $2d$-regular graph of odd order, with $d\ge 1$,
has a rooted decomposition into $d$ odd closed trails.
\end{conjecture1*}

\begin{figure}[htbp]
\centerline{\includegraphics[width=20em]{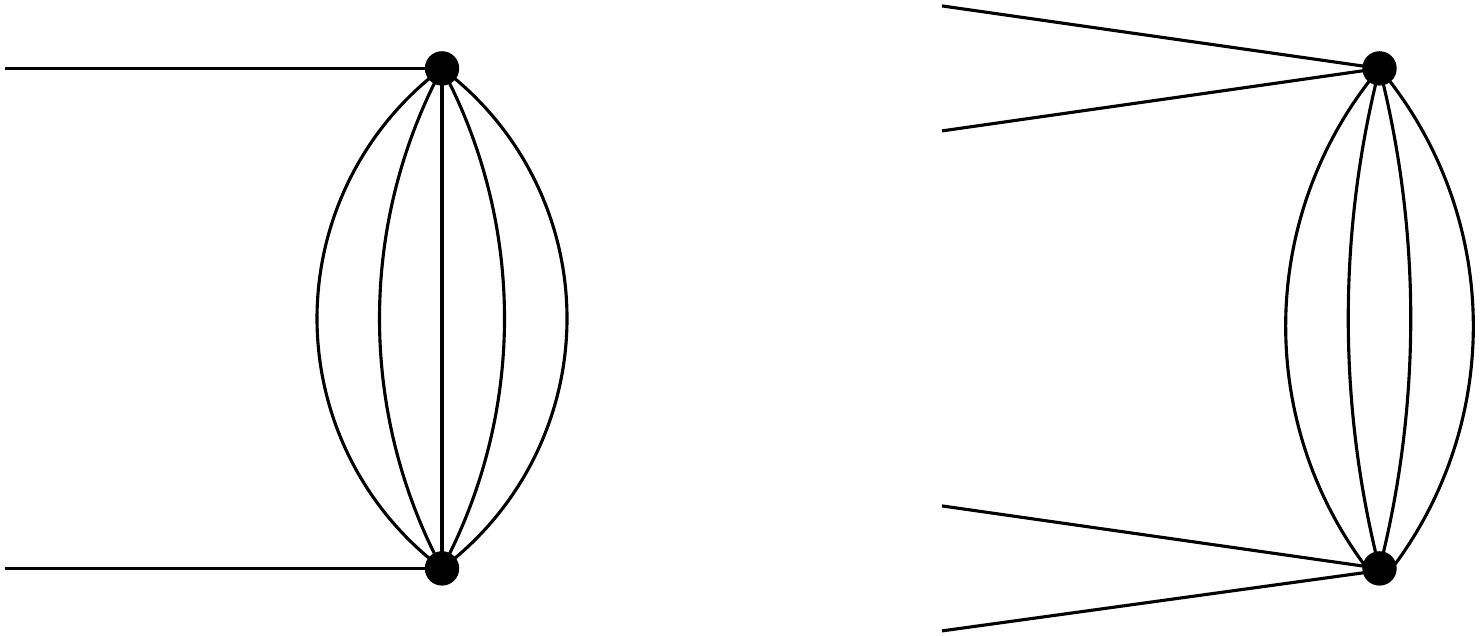}}
  \caption{None of these vertices can be a root of a decomposition into three odd trails}
  \label{fig:non-root}
\end{figure}

If a graph admits a rooted decomposition into $d$ odd closed
trails, then there is an obvious question about the
distribution of roots within the graph. Simple example show
that in general the root cannot be chosen arbitrarily even in
regular graphs. For example, none of the vertices depicted in
Figure~\ref{fig:non-root} can be a root of a decomposition into
three odd closed trails.

However, it seems plausible, that in a $4$-edge-connected
$4$-regular graph of odd order every vertex can be the root of
a decomposition into two odd closed trails. Similarly, it is
conceivable that for every vertex of a $6$-edge-connected
$6$-regular there exists a rooted decomposition into three odd
closed trails with root at that vertex. More generally, we
propose the following conjecture.

\begin{conjecture2*}
In a $2d$-edge-connected $2d$-regular graph of odd order, with
$d\ge 1$, every vertex is a root of some decomposition into $d$
odd closed trails.
\end{conjecture2*}

The rest of our paper is organised as follows. In the next
section we assemble the basic definitions needed in this paper.
In Section~\ref{sec:odddec} we begin the study of
decompositions into odd closed trails in general and derive a
necessary and sufficient condition for their existence. In
Section~\ref{sec:rooted} we introduce rooted decompositions and
prove Theorem~1. The proof of Theorem~2 is divided into two
parts which are contained in Sections~\ref{sec:proof-6-conn}
and~\ref{sec:proof-reduct}, respectively. In the final section
we briefly discuss a relationship between $3$-odd
decompositions and nowhere-zero integer $3$-flows
and  derive a necessary and sufficient
condition in for an eulerian graph to admit a rooted $3$-odd
decomposition.

\section{Preliminaries}\label{sec:prelim}

All graphs considered in this paper are finite and may have
multiple edges and loops. The graph consisting of a single
vertex and $d$ loops will be called the \textit{bouquet of $d$
circles} and will be denoted by $B_d$.

We often write $e=uv$ for an edge with end-vertices $u$ and
$v$, but this notation does not exclude the possibility that
$u=v$ or that there is another edge $f$ with the same
end-vertices as $e$.

An \textit{eulerian graph} has all vertices of even degree and
is always connected. A \textit{semi-eulerian graph} is a
connected graph with at most two vertices of odd degree. A
\textit{$u$-$v$-eulerian graph} is a connected graph where all
vertices except possibly $u$ and $v$ have an even degree; if
$u=v$, then the graph is eulerian.

Throughout this paper we make extensive use of various types of
walks, trails, paths, and their segments. In accordance with
the adopted terminology, a \textit{trail} is a non-empty
sequence $W=v_0e_1v_1e_2v_2\dots v_{k-1}e_kv_k$ whose terms are
alternately vertices and edges such that each edge $e_i$ joins
the vertex $v_{i-1}$ to the vertex $v_i$, and all edge terms
are distinct. More specifically, $W$ is a
\textit{$v_0$-$v_k$-trail}. If $v_0=v_k$, we say that the trail
is \textit{closed}. A \textit{path} is a trail in which all
vertex terms are distinct. A \textit{circuit} is a closed trail
in which all inner vertex terms are distinct. If $W_1$ is a
$u$-$v$-trail and $W_2$ is a $v$-$w$-trail, then $W_1W_2$
denotes the $u$-$w$-trail obtained by first traversing $W_1$
and then $W_2$. For subgraphs $H$ and $K$ of $G$ we define a
\textit{$K$-$H$-path} in $G$ as a $u$-$v$-path where the vertex
$u$ is in $K$, the vertex $v$ is in $H$, and all other vertices
lie outside $K\cup H$.

If $\{V_1,V_2\}$ is a partition of the vertex set of $G$, then
the set of all edges having an end-vertex in both partition
sets is called a \textit{cut} in $G$. A cut of size $n$ is an
\textit{$n$-edge-cut}. Recall that a graph $G$ is
\textit{$k$-edge-connected} if the removal of fewer than $k$
edges leaves $G$ connected. The \textit{edge-connectivity}
$\lambda(G)$ of $G$ is the largest integer $k$ for which $G$ is
$k$-edge-connected. In order to have the parameter $\lambda(G)$
always finite we make an exception from the definition when $G$
has a single vertex. In this case $G$ coincides with $B_d$ for
some $d\ge 0$, and we set $\lambda(B_d)=2d$. Note that if $G$
is eulerian, then $\lambda(G)$ is an even number. We will be
using this fact throughout the paper without mention.

Finally, if $G$ is a graph and $H$ and $K$ are subgraphs of
$G$, we let $H-K$ denote the subgraph of $H$ obtained by the
removal of the edges of $K$.

\section{Odd eulerian decompositions}\label{sec:odddec}

A \textit{decomposition} of a graph $G$ is a set
$\mathcal{D}=\{G_1,G_2,\ldots,G_k\}$ of subgraphs of $G$ whose
edge sets partition the edge set of $G$. The subgraphs $G_i$
constituting the decomposition will be called the
\textit{constituents} of $\mathcal{D}$. A decomposition
$\mathcal{D}$ will be called \textit{eulerian} if each
constituent is an eulerian subgraph; equivalently, if each constituent is an closed trail.
Clearly, a connected graph
that admits an eulerian decomposition must itself be eulerian.

The purpose of the present paper is to study eulerian
decompositions where each constituent $G_i$ has an odd number
of edges.  Such a decomposition will be called \textit{odd}.
More specifically, a decomposition
$\mathcal{D}=\{G_1,G_2,\ldots,G_k\}$ of an eulerian graph $G$
will be called \textit{$k$-odd} if each $G_i$ is eulerian and
has an odd number of edges.

Our first theorem characterises eulerian graphs that admit a
$k$-odd decomposition for a fixed integer $k\ge 1$.

\begin{theorem}\label{thm:k-odd}
Let $G$ be an eulerian graph and let $k$ be a positive integer
with $k\equiv |E(G)|\pmod{2}$. Then $G$ admits a $k$-odd
decomposition if and only if $G$ contains at least $k$ pairwise
edge-disjoint odd circuits.
\end{theorem}

\begin{proof}
To prove the necessity, let $\{G_1,G_2,\ldots,G_k\}$ be a
$k$-odd decomposition of $G$. Each $G_i$ has an eulerian trail
of odd length and since each closed walk of odd length contains
an odd circuit, $G_i$ contains an odd circuit $C_i$. Hence
$\{C_1,C_2,\ldots, C_k\}$ is a set of $k$ pairwise
edge-disjoint circuits in $G$.

For the converse, let $\{C_1,C_2,\ldots, C_k\}$ be an arbitrary
set of $k$ pairwise edge-disjoint circuits of $G$. Since each
component of $G-(\bigcup_i C_i)$ is eulerian, $G$ admits a
circuit decomposition $\mathcal{K}$ that includes all the
circuits $C_1,C_2,\ldots, C_k$. Let us consider the
intersection graph $J(\mathcal{K})$ of~$\mathcal{K}$; its
vertices are the elements of $\mathcal{K}$ and edges join pairs
of elements that have a vertex of $G$ in common. Since $G$ is
connected, so is $J(\mathcal{K})$.

It is obvious that every connected subgraph of
$J(\mathcal{K})$, with vertex set a subset
$\mathcal{L}\subseteq\mathcal{K}$, uniquely determines an
eulerian subgraph of $G$. The latter subgraph will have an odd
number of edges whenever $\mathcal{L}$ contains an odd number
of odd circuits. Thus to finish the proof it is enough to show
that $\mathcal{K}$ can be partitioned into $k$ subsets, each
containing an odd number of odd circuits and each inducing a
connected subgraph of~$J(\mathcal{K})$. In fact, we may assume
that $J(\mathcal{K})$ is a tree as the general case follows
immediately with the partition of $\mathcal{K}$ obtained from a
spanning tree of $J(\mathcal{K})$.

If $\mathcal{B}$ is the set of all odd circuits from
$\mathcal{K}$, then $|\mathcal{B}|\equiv |E(G)|\equiv
k\pmod{2}$ and $|\mathcal{B}|\ge k$. Thus we may view
$J(\mathcal{K})$ as a tree $T$ having a distinguished set $B$
of vertices such that $|B|\ge k$ and $|B|\equiv k\pmod{2}$. In
this terminology, it remains to prove the following.

\medskip\noindent
Claim 1. \textit{Let $T$ be a tree, $k\ge 1$ an integer, and
$B\subseteq V(T)$ a subset with $|B|\ge k$ vertices. If
$|B|\equiv k\pmod{2}$, then the vertex set of $T$ can be
partitioned into $k$ pairwise disjoint subsets $V_1,V_2,\ldots,
V_k$ such that each $V_i$ contains an odd number of vertices
from $B$ and induces a subtree of $T$.}

\medskip\noindent
Proof of Claim~1. We proceed by induction on $k$, and for every
fixed $k$ by induction on the number of vertices of $T$. If
$k=1$, then the conclusion is vacuously true for every tree $T$
and every subset $B\subseteq V(T)$ with an odd number of
vertices. If $k=2$, take the largest subtree $T_1$ of $T$ with
an odd number of vertices from $B$, let $V_1$ be the vertex set
of $T_1$, and let $V_2=V(T)-V_1$. Clearly, the partition
$\{V_1,V_2\}$ fulfils the conditions of the claim.

For the induction step assume that $T$ has $k\ge 3$ vertices.
Let $T$ be any tree with a subset $B\subseteq V(T)$ such that
$|B|\ge k$ and $|B|\equiv k\pmod{2}$. Clearly, $T$ has at least
$k$ vertices. If $T$ has exactly $k$ vertices, then $B=V(T)$,
and the partition into singletons is the sought partition of
$V(T)$. Assume now that $T$ has more than $k$ vertices. At
least two vertices of $T$ are leaves, say $u$ and $v$. There
are two cases to consider.

\medskip\noindent
\textbf{Case 1.} Both $u$ and $v$ belong to $B$. Consider the
tree $T'=T-\{u,v\}$ and the set $B'=B-\{u,v\}$. Since $B'$ has
at least $k-2$ vertices, we can apply the induction hypothesis
to $T'$ and $B'$ for $k'=k-2$. From the induction hypothesis we
get a partition $\{V_1,V_2,\ldots,V_{k-2}\}$ of $V(T')$ such
that each $V_i$ contains an odd number of vertices from $B'$
and induces a subtree of $T'$. Then
$\{V_1,V_2,\ldots,V_{k-2},\{u\},\{v\}\}$ is the required
partition for~$T$.

\medskip\noindent
\textbf{Case 2.} One of $u$ and $v$, say $u$, does not belong
to $B$. In this case we set $T'=T-u$ and $B'=B$, and apply the
induction hypothesis to $T'$ and $B'$ for $k'=k$. We conclude
that $V(T')$ has a partition $\{V_1,V_2,\ldots,V_k\}$ such that
each $V_i$ contains an odd number of vertices from $B'$ and
induces a subtree of $T'$. One of the partition sets, say
$V_1$, contains a neighbour of $u$. Then
$\{V_1\cup\{u\},V_2,\ldots,V_k\}$ is the required partition for
$T$.

\medskip
This concludes the induction step and establishes the claim as
well as the theorem.
\end{proof}

\begin{corollary}\label{cor:2k-reg}
Let $G$ be a connected $2d$-regular graph of odd order with
$d\ge 1$. Then $G$ admits an $k$-odd decomposition for each
$k\in\{1,2,\dots, d\}$ such that $k\equiv d\pmod{2}$.
\end{corollary}

\begin{proof}
By Petersen's $2$-factor theorem, $G$ can be decomposed into
$d$ pairwise edge-disjoint $2$-factors (see, for example,
\cite[Corollary 2.1.5]{Diestel}). Since $G$ has an odd order, each
$2$-factor contains an odd circuit, so $G$ contains at least
$d$ pairwise edge-disjoint circuits. For each $k\le d$ such
that $k\equiv d\pmod{2}$ we have $k\equiv |E(G)|\pmod{2}$, and
the conclusion now immediately follows from
Theorem~\ref{thm:k-odd}.
\end{proof}

\section{Rooted odd decompositions}\label{sec:rooted}

A $k$-odd decomposition $\mathcal{D}=\{G_1,G_2,\ldots,G_k\}$ of
an eulerian graph $G$ is \textit{rooted} if there is a vertex
$v$ of $G$, called a \textit{root} of $\mathcal{D}$, such that
each constituent $G_i$ contains~$v$.  Clearly, a graph $G$
admits a rooted $k$-odd decomposition if and only if it
contains an eulerian trail $T$ and a vertex $v$ that divides
$T$ into $k$ closed subtrails $T_1, T_2,\ldots, T_k$ based at
$v$, each of odd length.

Finding a rooted $k$-odd decomposition is a significantly more
difficult task than just finding any $k$-odd decomposition,
especially as $k$ increases. Nevertheless, every $2$-odd
decomposition is automatically rooted, so
Theorem~\ref{thm:k-odd} also provides a characterisation of
graphs that admit a $2$-odd decomposition. The following
theorem somewhat simplifies the condition and is provided with
a short independent proof.

\begin{theorem}\label{thm:2bly-odd}
An eulerian graph has a rooted $2$-odd decomposition if and
only if it is non-bipartite and has an even number of edges.
\end{theorem}

\begin{proof}
Each odd closed trail contains an odd circuit, so the necessary
condition follows immediately. For the converse, let $G$ be a
non-bipartite eulerian graph with an even number of edges. Then
$G$ contains an odd circuit and hence an eulerian subgraph with
an odd number of edges. Let $C$ be one that has the maximal
number of edges. Since $G$ has an even number of edges, $C$ is
a proper subgraph of $G$. Each component of $G-C$ is an
eulerian subgraph of $G$ sharing a vertex with $C$. Since $C$
is maximal, there is only one non-trivial component, and it
must have an odd number of edges. It follows that $\{C,G-C\}$
is a rooted $2$-odd decomposition of~$G$.
\end{proof}

We now derive Theorem~\ref{thm:main1} as a consequence of
Theorem~\ref{thm:2bly-odd}.

\begin{proof}[Proof of Theorem~\ref{thm:main1}]
Let $G$ be a connected $4$-regular graph of odd order. Clearly,
$G$ has an even number of edges. Furthermore, $G$ cannot be
bipartite, because the partite sets of a connected regular
bipartite graph have the same size, and hence such a graph has
an even number of vertices. The result now follows from
Theorem~\ref{thm:2bly-odd}.
\end{proof}

Having characterised graphs that admit a rooted $2$-odd
decomposition we can proceed to rooted $3$-odd decompositions.
The situation here is much more complicated because there exist
eulerian graphs that admit a $3$-odd decomposition, but not a
rooted one. One example, which can be easily extended to an
infinite family, is displayed in Figure~\ref{fig:3neg}. It
follows that a structural characterisation of graphs that admit
a rooted $3$-odd decomposition may be difficult. Our Theorem~2 shows
that having an odd number of vertices (or edges) is
a sufficient condition for a connected $6$-regular graph
to have a rooted $3$-odd decomposition. Trivially, this condition is also necessary.
The proof heavily
depends on connectivity arguments and will be performed in two
steps. First, in Section~\ref{sec:proof-6-conn}, we prove a
special case of Theorem~2 for $6$-edge-connected graphs. The
general case will be treated in Section~\ref{sec:proof-reduct}
after some additional preparation.

\begin{figure}[htbp]
\centerline{\includegraphics[width=20em]{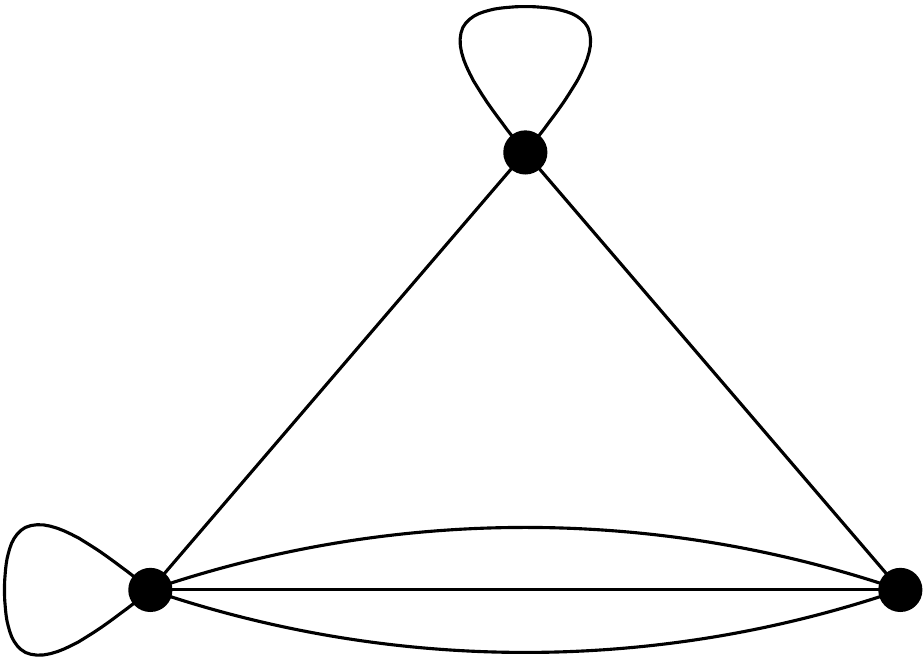}}
  \caption{An eulerian graph having no rooted $3$-odd decomposition. }
  \label{fig:3neg}
\end{figure}

\section{Proof of Theorem~2: The 6-connected case}\label{sec:proof-6-conn}

In this section we prove that every $6$-edge-connected
$6$-regular graph of odd order has a rooted $3$-odd
decomposition. To make the proof easier we pass from ordinary
graphs to signed graphs and prove a natural analogue of the
required statement for signed graphs. The statement for
unsigned graphs will follow as a trivial consequence of the
signed graph version.

Recall that a \textit{signed graph} is a graph $G$ together
with a mapping, called the \textit{signature} of $G$, which
assigns $+1$ or $-1$ to each edge. An edge receiving value $+1$
is said to be \textit{positive} while one with value $-1$ is
said to be \textit{negative}. The sign of each edge will
usually be known from the immediate context, therefore no
special notation for the signature will be required.

The signature of a signed graph is a means of introducing the
concept of balance, which is more important than the signature
itself. A circuit of a signed graph is said to be
\textit{balanced} if it contains an even number of negative
edges, and is \textit{unbalanced} otherwise. A signed graph in
which all circuits are balanced is itself called
\textit{balanced}; an \textit{unbalanced} signed graph is one
that contains at least one unbalanced circuit. In general, the
essence of any signed graph is constituted by the list of all
balanced circuits. Two signed graphs with the same underlying
graph are therefore considered to be \textit{identical} if
their lists of balanced circuits coincide. The corresponding
signatures are called \textit{equivalent}.

There is a convenient way of turning one signature into an
equivalent one. Let $G$ be a signed graph and let $U$ be a set
of vertices of $G$. If we change the sign of each edge with
exactly one end in $U$, then the product of signs on every
circuit does not change and hence the new signature is
equivalent to the previous one. This operation is called
\textit{switching} at $U$. Note that switching the signature at
$U$ has the same effect as switching at all the vertices of $U$
in a succession. It is easy to see that by successive vertex
switching we can turn any spanning tree of $G$ into an
all-positive subgraph. This fact readily implies that two
signatures are equivalent if and only if they are
\textit{switching-equivalent}, that is, if they can be
transformed into each other by a sequence of vertex switchings
\cite[Proposition~3.2]{Zaslav1}. In particular, a signed graph
is balanced if and only if its signature is equivalent to the
all-positive signature. For a more detailed introduction to
signed graphs we refer the reader to Zaslavsky \cite{Zaslav1}.

Observe that switching the signature of a signed eulerian graph
does not change the parity of the number of negative edges.
Therefore all signed eulerian graphs fall into two natural
subclasses depending on whether the number of negative edges is
even or odd. Accordingly, a signed eulerian graph $G$ will be
called \textit{even} if it has an even of negative edges,
otherwise $G$ will be called
\textit{odd}. It is easy to see that even eulerian graphs can
be balanced as well as unbalanced. In contrast, odd eulerian graphs
are necessarily unbalanced.

A decomposition $\mathcal{D}=\{G_1,G_2,\ldots,\penalty0 G_k\}$
of a signed eulerian graph $G$ will be called \textit{odd}, or
more specifically \textit{$k$-odd}, if each $G_i$ is an odd
eulerian signed subgraph of $G$. It may be useful to realise that a
decomposition of an unsigned graph $G$ is odd if and only if it
is odd for the signed graph obtained from $G$ by assigning $-1$
to each edge.

We proceed to the main result of this section which gives a
sufficient condition for a $6$-regular signed graph to have a
rooted $3$-odd decomposition. Clearly, every signed eulerian
graph that admits a $3$-odd decomposition must have an odd
number of negative edges, and must contain at least three
pairwise edge-disjoint unbalanced circuits, one for each
constituent. We show that for $6$-edge-connected $6$-regular
graphs these necessary conditions are also sufficient. In fact,
we can show that it is enough to require two edge-disjoint
unbalanced circuits -- the parity of the number of negative
edges will ensure the existence of three.

\begin{lemma}\label{lemma:3_unbal_circuits}
Let $G$ be a signed eulerian graph with an odd number of
negative edges which contains two edge-disjoint unbalanced
circuits. Then $G$ contains at least three pairwise
edge-disjoint unbalanced circuits.
\end{lemma}

\begin{proof}
Take two edge-disjoint unbalanced circuits $C_1$ and $C_2$ of
$G$ and form the signed graph $G'=G-(C_1\cup C_2)$. The total
number of negative edges in components of $G'$ is odd, so $G'$
has a component $K$ with an odd number of negative edges. Since
$K$ is eulerian, it contains an unbalanced circuit $C_3$. The
circuits $C_1$, $C_2$, and $C_3$ are obviously pairwise
edge-disjoint, as required.
\end{proof}

Now we are ready for the main result of this section.

\begin{theorem}\label{thm:sign6-reg}
Let $G$ be a $6$-edge-connected $6$-regular signed graph with
an odd number of negative edges which contains two edge
disjoint unbalanced circuits. Then $G$ has a rooted $3$-odd
decomposition.
\end{theorem}

\begin{proof}
We prove the result by induction on the number of vertices. If
$G$ has a single vertex, then $G$ must be a bouquet of three
negative loops, and the conclusion for $G$ is clearly holds.

For the induction step, and throughout the rest of the proof,
let $G$ be a $6$-edge-connected $6$-regular signed graph with
an odd number of negative edges which contains two edge
disjoint unbalanced circuits on $n\ge 2$ vertices.

Consider an arbitrary vertex $v$ of $G$ and let $e_1, e_2,
\ldots, e_6$ be the edges incident with $v$ listed in a certain
fixed order, and let $v_i$ denote the other end of $e_i$. Let
us form a graph $G'$ of order $n-1$ by removing $v$ from $G$
and by adding three new edges $e_i'=v_iv_{i+3}$ for
$i\in\{1,2,3\}$ to $G-v$; note that by doing this we may
introduce parallel edges and loops. We define the signature for
$G'$ in a natural way: the sign of each edge $v_iv_{i+3}$ will
be obtained by multiplying the sign of $v_iv$ with the sign of
$vv_{i+3}$, the signs of all other edges being directly
inherited from $G$. We will say that the signed graph $G'$ is
obtained by \textit{splitting off} the vertex $v$ from $G$.

In 1992, Frank \cite[Theorem A$'$]{Frank}, generalising an
earlier result of Lov\'asz \cite{Lovasz}, proved that every
vertex of even degree $d\ge 4$ in a $2$-edge-connected graph
$K$ can be split off in a similar manner as defined above to
produce a graph $K'$ that has the same edge-connectivity as
$K$; that is, $\lambda(K')=\lambda(K)$. Using this fact we can
prove the following.

\bigskip\noindent
Claim 1. \textit{Every vertex of $G$ can be split in such a way
that the resulting signed graph $G'$ is $6$-edge-connected,
$6$-regular, and has an odd number of negative edges.}

\medskip\noindent
Proof of Claim 1. Given a vertex $v$ of $G$, let us perform
splitting in the way guaranteed by the result of Frank
\cite{Frank}. It is obvious that $G'$ is $6$-regular and
$6$-edge-connected. Observe that the signature for $G'$ has
been defined in such a way that the number of negative edges in
$G'$ has the same parity as that of $G'$. Therefore $G'$ is odd
and thus has all the properties stated. This proves Claim~1.

\bigskip\noindent Claim 2. \textit{Let $G'$ be a connected
$6$-regular signed graph with an odd number negative edges
obtained from $G$ by splitting off a vertex $v$ of $G$. If $G'$
admits a rooted $3$-odd decomposition, then so does $G$.}

\medskip\noindent
Proof of Claim~2. The graph $G$ can be reconstructed from $G'$ by
first subdividing each edge $e_i'=v_iv_{i+3}$, where
$i\in\{1,2,3\}$, with a new vertex, then by identifying the
three new vertices into one -- the vertex $v$ -- and by
reinstating the original signature of $G$ on the newly formed
edges of $G$. Let $\{G_1',G_2',G_1'\}$ be a rooted $3$-odd
decomposition of $G'$. The process in which $G$ arises from
$G'$ produces from each $G_i'$ an eulerian subgraph $G_i$ of
$G$ with the parity of the number of negative edges preserved.
Therefore $\{G_1,G_2,G_3\}$ is a rooted $3$-odd decomposition
of $G$, and Claim~2 is proved.

\medskip\noindent
To finish the proof of the theorem it suffices to prove the
following claim.

\bigskip\noindent
Claim 3. \textit{At least one of the following statements holds
for $G$:}
\begin{itemize}
\item[(1)] \textit{$G$ admits a rooted $3$-odd
    decomposition.}
\item[(2)] \textit{$G$ contains a vertex that can be split
    in such a way that the resulting signed graph $G'$ is
    $6$-edge-connected $6$-regular signed graph with an odd
number of negative edges which contains two edge disjoint
unbalanced circuits.}
\end{itemize}

\noindent Proof of Claim 3. To prove the claim it is enough to
show that either $G$ admits a rooted $3$-odd decomposition or
$G$ has a vertex such that its splitting off from $G$ in
accordance with Claim~1 produces a signed graph with two
edge-disjoint unbalanced circuit.
The remaining conditions are automatically fulfilled.

In order to do it we first recall that, by
Lemma~\ref{lemma:3_unbal_circuits}, $G$ contains three
edge-disjoint unbalanced circuits $C_1$, $C_2$, and $C_3$. We
will analyse the possible positions of $\{C_1,C_2,C_3\}$ within
$G$ and in each case we show that either (1) or (2) holds.

If $G$ contains a vertex $v$ that belongs to at most one of
$C_1$, $C_2$, and $C_3$, we split $v$ according to Claim~1. The
resulting graph $G'$ is $6$-regular, $6$-edge-connected,
retains at least two edge-disjoint unbalanced circuits, and
has an odd number of negative edges. In other words, (2) holds. Thus we may assume that
each vertex of $G$ is contained in at least two circuits from $\{C_1, C_2, C_3\}$.
Consider the subgraph $H$ of $G$ obtained from $G$
by removing all edges of $C_1\cup C_2\cup C_3$ and by deleting
isolated vertices that may result. Since $G$ is $6$-regular,
each component of $H$ is a circuit whose vertices lie in
$C_1\cup C_2\cup C_3$. Moreover, $G=C_1\cup C_2\cup C_3\cup H$.

Next, assume that $G$ contains a vertex $v$ belonging to all
three circuits $C_1$, $C_2$, and~$C_3$. If $H$ is balanced,
then for each $i\in\{1,2,3\}$ we can form a subgraph $G_i$ by
taking the union of $C_i$ with some of the circuits of $H$ that
share a vertex with $C_i$. Clearly, each $G_i$ is an odd
eulerian signed graph. Furthermore, if each circuit of $H$ is absorbed
into precisely one $G_i$, then $\{G_1,G_2,G_3\}$ becomes a
$3$-odd decomposition of $G$ with root at $v$. This
verifies~(1). If $H$ is unbalanced, $H$ contains a pair of
disjoint unbalanced circuits $D_1$ and $D_2$ because the total
number of negative edges in $H$ is even. Clearly, both of $D_1$
and $D_2$ are disjoint from~$v$. We now split $v$ according to
Claim~1, so $D_1$ and $D_2$ will be inherited into the
resulting graph $G'$. Taking into account Claim~1 we see that
$G'$ has all the required properties. Hence (2) is satisfied.

For the rest of the proof we may assume that every vertex of
$G$ lies on precisely two of the circuits $C_1$, $C_2$, and
$C_3$. In particular, $H$ is a $2$-factor. By parity, $H$
contains an even number of unbalanced circuits. If $H$ contains
at least two unbalanced circuits, say $D_1$ and $D_2$. Take any
vertex $v$ of $D_1$ and split it in accordance with Claim~1 to
produce a graph $G'$. There exists exactly one circuit
$C_j\in\{C_1, C_2, C_3\}$ that does not contain $v$.
Furthermore, $D_2$ also does not contain $v$ because $D_1\cap
D_2=\emptyset$. Therefore both $C_j$ and $D_2$ are inherited to
$G'$, which establishes (2).

Thus we may assume that $H$ is balanced. Recall that each
vertex of $G$ is of one of three types, depending on which pair
of circuits from $\{C_1, C_2, C_3\}$ it belongs to. If $H$
contains a circuit $B$ with vertices of different types, we
construct a $3$-odd decomposition of $G$ as follows. Without
loss of generality we may assume that $B$ has a vertex $v$ that
belongs to $C_1\cap C_2$ and a vertex $w$ that belongs to
$C_1\cap C_3$. In this situation $C_1$, $B\cup C_2$, and $C_3$
are three closed trails that share the vertex~$v$. We now
extend $\{C_1, C_2, B\cup C_3\}$ to a decomposition
$\{G_1,G_2,G_3\}$ of $G$ by adding each component $K$ of $H$ to
a member of $\{C_1,C_2, B\cup C_3\}$ which is intersected by
$K$. It is easy to see that $\{G_1,G_2,G_3\}$ is indeed a
$3$-odd decomposition of $G$ rooted at $v$, which verifies (1).

We are left with the case where $H$ is balanced and each
component of $H$ contains vertices of the same type. Pick an
arbitrary vertex $v$ of $G$. It belongs to exactly two circuits
from $\{C_1,C_2,C_3\}$, say to $C_1$ and $C_2$. There is an
edge $e$ of $H$ incident with $v$. Since $G$ is
$6$-edge-connected, $e$ cannot be a loop, so $e=vw$ for some
vertex $w\ne v$. The vertex $w$ must have the same type as $v$,
so $w$ also belongs to $C_1\cap C_2$; in particular, $e$ is a
chord of both $C_1$ and $C_2$. Assume that $w$ is not a
neighbour of $v$ in at least one of $C_1$ and $C_2$, say in
$C_1$. Then $e$ divides $C_1$ into two edge-disjoint
$v$-$w$-paths $P$ and $Q$ of length at least~$2$. Since $C_1$
is unbalanced, exactly one of the circuits $Pe$ and $eQ$ is
unbalanced, say $Pe$. Now $Pe$, $C_2$, and $C_3$ are three
pairwise edge-disjoint unbalanced circuits. Furthermore, any
inner vertex $u$ of $Q$ belongs to exactly two members of
$\{C_1,C_2,C_3\}$ one of which is $C_1$. It follows that $u$
belongs to only one of the circuits $\{Pe, C_2, C_3\}$, and by
the case already treated this vertex can be split off to fulfil
(2). Hence, if $G$ fails to satisfy (1) or (2), then $w$ is a
neighbour of $v$ in both $C_1$ and $C_2$. In such a case,
however, we can repeat the same reasoning for $w$ in place of
$v$, and so on all around the same component of~$H$. As a
result of this consideration we deduce that $G=C_1\cup C_2\cup
H$, which is absurd. This contradiction shows that either (1)
or (2) holds for $G$. This completes the proof of Claim~3 as
well as that of the theorem.
\end{proof}

\begin{corollary}\label{cor:3-odd_6-conn}
Every $6$-edge-connected $6$-regular graph of odd order has a
rooted  $3$-odd decomposition.
\end{corollary}
\begin{proof}
Let us endow $G$ with the all-negative signature. Since $G$ has
an odd number of edges, as a signed graph $G$ is odd.
Petersen's $2$-factor theorem further implies that $G$ can
be decomposed into three pairwise edge-disjoint $2$-factors,
and since the number of vertices of $G$ is odd, each of these
2-factors contains an odd circuit. Thus $G$ contains at least three
unbalanced circuits. Theorem~\ref{thm:sign6-reg} now yields that
under the all-negative signature $G$ has a rooted $3$-odd
decomposition. This decomposition is clearly $3$-odd also in the
unsigned sense, and the result is proved.
\end{proof}

\begin{remark}
{\rm The assumption of Theorem~\ref{thm:sign6-reg} requiring a
$6$-regular graph $G$ to be $6$-edge-connected is essential and
cannot be relaxed. Figure~\ref{fig:bez_rooted} displays
odd signed $6$-regular graphs $G_1$ and $G_2$ with $\lambda(G_1)=2$
and $\lambda(G_2)=4$ neither of which admits a rooted $3$-odd decomposition
(edges not labelled are positive). This shows that Theorem~2
does not directly generalise to signed $6$-regular graphs without
additional assumptions.}
\end{remark}

\begin{figure}[htbp]
\centerline{\includegraphics[width=30em]{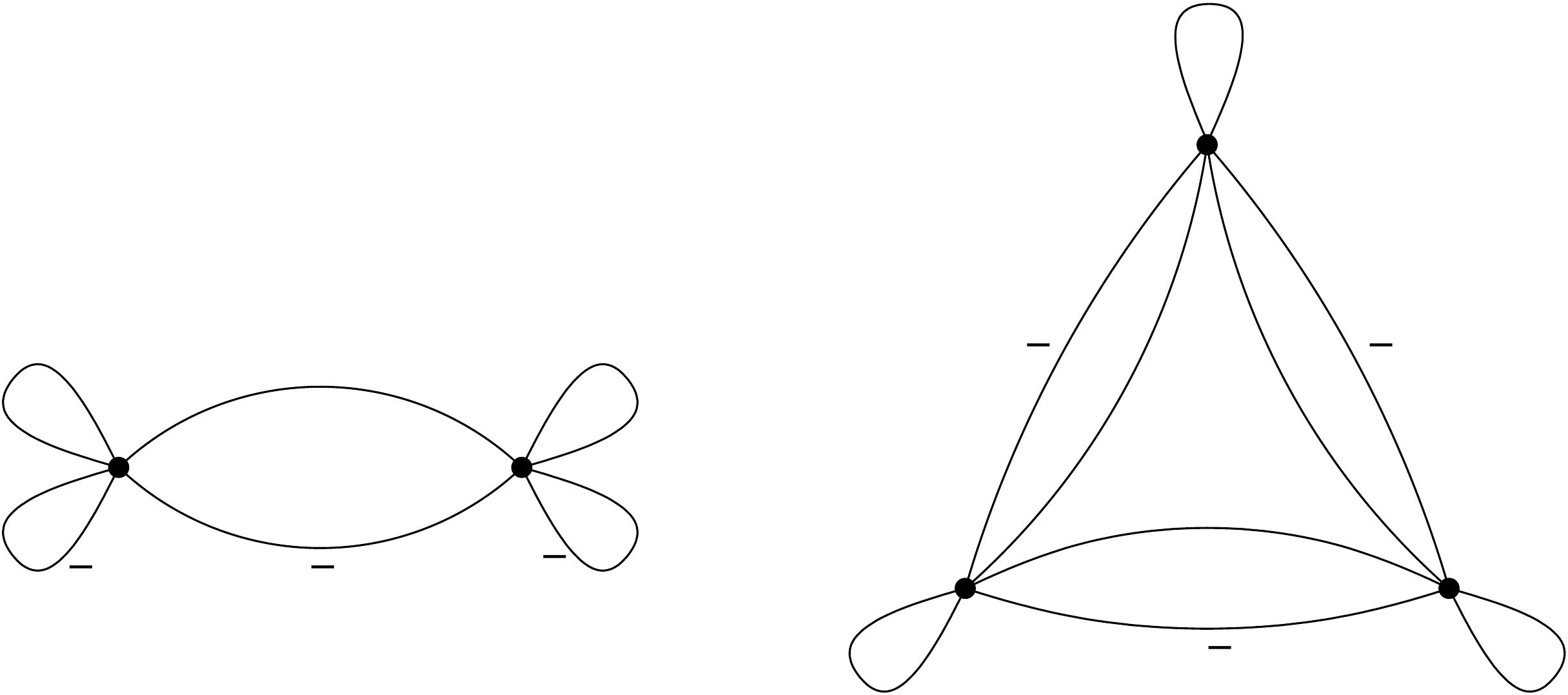}}
  \caption{Connected $5$-regular signed graph with no $3$-odd decomposition.}
  \label{fig:bez_rooted}
\end{figure}

\begin{remark}
\rm{The formulation of Theorem~\ref{thm:sign6-reg} can be
improved by using concepts and results from \cite{euler}.
Namely, the assumption requiring a signed eulerian graph $G$ to
have at least two edge-disjoint unbalanced circuits can be
replaced by the condition that $G$ cannot be turned into a
balanced graph by deleting a single edge. To be more precise,
let us call a signed graph $G$ \textit{tightly unbalanced} if
it is unbalanced and contains an edge $e$ such that $G-e$ is
balanced. If $G$ is unbalanced but not tightly unbalanced we
say that it is \textit{amply unbalanced}. Amply and tightly
unbalanced signed graphs are important classes of signed graphs
which naturally appear in the study of various problems, see
for example \cite{cycosig, euler, sigmaxgen}. In
\cite[Corollary~3.4]{euler} it is shown that an eulerian signed
graph admits a nowhere-zero integer flow if and only if it is
amply unbalanced. Further, the equivalence
(a)$\Leftrightarrow$(c) in Theorem~4.2 from \cite{euler} states
that a signed eulerian graph contains two edge-disjoint
unbalanced circuits if and only if it is amply unbalanced. The
proof of this equivalence is non-trivial.}
\end{remark}

\section{Proof of Theorem~2: The general case}\label{sec:proof-reduct}

In this section we prove Theorem~2. We begin with two simple
lemmas.

\begin{lemma}\label{lemma:oddvertices}
If a connected graph has exactly $2k$ vertices of odd degree,
then it contains a set of $k$ pairwise edge-disjoint paths
whose ends cover all odd-degree vertices.
\end{lemma}

\begin{proof}
This is an immediate consequence of the following classical
result due to Listing and Lucas (see K\"onig \cite[Satz~4,
p.~22]{Konig}): Every connected graph $G$ with exactly $2k$
vertices of odd degree can be decomposed into $k$ open trails,
and each such decomposition contains at least $k$ open trails.
\end{proof}

Consider a vertex $v$ of a graph $G$ and a sequence
$(v_1,v_2,\ldots,v_k)$ of vertices of $G$ not containing $v$;
note that we permit $v_i=v_j$ for $i\ne j$. A
\textit{$v$-$(v_1,v_2,\ldots,v_k)$-fan} in $G$ is a collection
$\mathcal{F}=\{F_1,F_2,\ldots,F_k\}$ of $k$ edge-disjoint paths
such that the path $F_i$ joins $v$ to $v_i$ for each
$i\in\{1,2,\ldots,k\}$.

\begin{lemma}\label{lemma:kconn-kfan}
Let $G$ be a $k$-edge-connected graph. Then for every vertex
$v$ and an arbitrary sequence $v_1,v_2,\ldots,v_k$ of vertices
of $G$ there exists a $v$-$(v_1,v_2,\ldots,v_k)$-fan in $G$.
\end{lemma}
\begin{proof}
Consider the graph $G^{+}$ arising from $G$ by adding a new
vertex $w$ together with $k$ new edges $wv_1,wv_2,\ldots,
wv_k$. We first show that $G^{+}$ is $k$-edge-connected.
Suppose not. Then $G^{+}$ contains an $m$-edge-cut $S$ with
$m<k$ separating $w$ from some vertex $z$ of $G$. At least one
of the edges incident with $w$ does not belong to $S$, say
$wv_t$. Since $G$ is $k$-edge-connected, it contains $k$
pairwise edge-disjoint $v_t$-$z$-paths. As $m<k$, one of the
paths, say $Q$, includes no edge from $S$. It follows that
$wv_tQ$ is a $w$-$z$-path in $G^{+}-S$, contradicting the
choice of $S$. To finish the proof observe that the
edge-connectivity version of Menger's Theorem implies that for
any vertex $w\ne v$ there exist $k$ pairwise edge-disjoint
$v$-$w$ paths in~$G^{+}$. The removal of $w$ from $G^{+}$
leaves in $G$ the required $k$ paths which form a
$v$-$(v_1,v_2,\ldots,v_k)$-fan in $G$.
\end{proof}

Now we are in position to prove Theorem~2.

\vskip3mm\noindent\textit{Proof of Theorem~2.} Let $G$ be a
smallest counterexample to Theorem~2. Below, in
Propositions~\ref{prop:2-cuts} and~\ref{prop:4-cuts}, we show
that $G$ contains neither a 2-edge-cut nor a 4-edge-cut and
therefore must be 6-edge-connected. However, by
Corollary~\ref{cor:3-odd_6-conn}, this is impossible. What remains is
to prove Propositions~\ref{prop:2-cuts}
and~\ref{prop:4-cuts}.

\begin{proposition}\label{prop:2-cuts}
A smallest counterexample to Theorem~2 must be
$4$-edge-connected.
\end{proposition}

\begin{proof}
Let $G$ be a smallest counterexample to Theorem~2. Suppose that
$G$ contains a 2-edge-cut $S$. Then $S$ separates $G$ into two
components $H$ and~$K$. One of them, say $H$, has an odd number
of vertices, and therefore an even number of edges.
Consequently, $K$ has an odd number of edges. Let $S=\{a_1b_1,
a_2b_2\}$ where $\{a_1, a_2\}\subseteq V(H)$ and $\{b_1,
b_2\}\subseteq V(K)$. Note that we do not exclude the
possibility that $a_1=a_2$ or $b_1=b_2$, or both. The graph
$H'=H+a_1a_2$ is a $6$-regular graph of odd order smaller
than~$G$. By the induction hypothesis, $H'$ has a rooted
$3$-odd decomposition $\mathcal{D}=\{H_1,H_2,H_3\}$ with root
at some vertex $v$ of $H$. One of the constituents of
$\mathcal{D}$, say $H_1$, contains the edge $a_1a_2$. By
setting $H_1^+=(H_1-a_1a_2)\cup S\cup K$ we get a rooted
$3$-odd decomposition $\{H_1^+,H_2,H_3\}$ of the entire~$G$
with root at the same vertex~$v$. This contradiction proves
that $G$ is $4$-edge-connected.
\end{proof}

\begin{proposition}\label{prop:4-cuts}
A smallest counterexample to Theorem~2 must be
$6$-edge-connected.
\end{proposition}

\begin{proof}
Again, let $G$ be a smallest counterexample to Theorem~2. By
Proposition~\ref{prop:2-cuts}, $G$ is $4$-edge-connected.
Suppose that $G$ contains a 4-edge-cut $S$. Then $G-S$ has two
components $H$ and $K$, one of which, say $H$, has an even
number of edges. It follows that $H$ has even order and that
$K$ has odd. Let $S=\{a_1b_1, a_2b_2, a_3b_3, a_4b_4\}$ with
$\{a_1,a_2,a_3,a_4\}\subseteq V(H)$ and
$\{b_1,b_2,b_3,b_4\}\subseteq V(K)$. Note that some vertices
within both $\{a_1,a_2,a_3,a_4\}$ and $\{b_1,b_2,b_3,b_4\}$ may
coincide. Set $A=\{a_1,a_2,a_3,a_4\}$.

To establish the result it suffices to show that $H$ can be
decomposed into two odd semi-eulerian subgraphs, an
$a_k$-$a_l$-eulerian subgraph $H_{k,l}$ and an
$a_m$-$a_n$-eulerian subgraph $H_{m,n}$, where
$\{k,l,m,n\}=\{1,2,3,4\}$. Having such a decomposition of~$H$,
we can construct a rooted $3$-odd decomposition of $G$ as
follows. We add the edges $b_kb_l$ and $b_mb_n$ to $K$ to form
a $6$-regular graph $G'$ of odd order. Since $H$ contains at
least two vertices, the order of $G'$ is smaller than that of
$G$. By the induction hypothesis, $G'$ has a rooted $3$-odd
decomposition $\mathcal{D}=\{K_1,K_2,K_3\}$ with root at some
vertex $v$ of $K$. We replace the edge $b_kb_l$ with $H_{k,l}$
and the edge $b_mb_n$ with $H_{m,n}$ in the corresponding
constituents of $\mathcal{D}$ thereby producing a rooted
$3$-odd decomposition of the entire~$G$ with root at the same
vertex -- a contradiction. What remains is to find the
subgraphs $H_{k,l}$ and $H_{m,n}$. We distinguish two cases
depending on whether $H$ is or is not bipartite.

\medskip
\textbf{Case 1.} The subgraph $H$ is bipartite. Let
$\{V_1,V_2\}$ be the bipartition of $H$. Let $x$ denote the
number of indices $i$ for which $a_i$ belongs to $V_1$;
obviously, the number of indices $i$ for which $a_i$ belongs to
$V_2$ is $4-x$. Since $G$ is bipartite, we have
$6|V_1|-x=6|V_2|-(4-x)$, which implies that $x\equiv2\pmod6$
and therefore $x=2$.

Without loss of generality we may assume that $a_1$ and $a_2$
lie in $V_1$ and $a_3$ and $a_4$ lie in $V_2$. (Note that the
vertices $a_1$ and $a_2$ may coincide as well as $a_3$ and
$a_4$ may coincide.) Since $H$ is connected, there exists an
$a_1$-$a_3$-path $P_{1,3}$ in $H$. Clearly $P_{1,3}$ has odd
length. The graph $H-P_{1,3}$ has precisely two vertices of odd
degree, namely $a_2$ and $a_4$. It follows that $a_2$ and $a_4$
lie in the same component of $H-P_{1,3}$, so there exists an
$a_2$-$a_4$-path $P_{2,4}$ in $H-P_{1,3}$, which again must
have an odd number of edges. The graph $H-(P_{1,3}\cup
P_{2,4})$ is bipartite and has all its vertices of even degree.
It follows that $H-(P_{1,3}\cup P_{2,4})$ can be decomposed
into a collection of even circuits; in particular, each
component of $H-(P_{1,3}\cup P_{2,4})$ has an even number of
edges. Moreover, each of these components is incident with at
least one of the paths $P_{1,3}$ and $P_{2,4}$. Thus we can add
each component of $H-(P_{1,3}\cup P_{2,4})$ to either $P_{1,3}$
or $P_{2,4}$ to produce a decomposition of $H$ into an
$a_1$-$a_3$-eulerian subgraph $H_{1,3}$ and
$a_2$-$a_4$-eulerian subgraph $H_{2,4}$. As the number of edges
in each of $H_{1,3}$ and $H_{2,4}$ is odd,
$\{H_{1,3},H_{2,4}\}$ is the required decomposition of $H$.

\medskip
{\bf Case 2.} The subgraph $H$ is not bipartite. As before, we
wish to construct suitable semi-eulerian subgraphs $H_{k,l}$
and $H_{m,n}$ that decompose $H$. To this end, the following
technical tool will be useful.

\medskip\noindent
Claim 2. \textit{Let $Y$ be an eulerian subgraph of $H$ and let
$B_Y$ be the union of all nontrivial components of $H-Y$ that
contain a vertex of $A$. Then $B_Y$ can be decomposed into two
semi-eulerian subgraphs $B_{k,l}$ and $B_{m,n}$ such that
$B_{k,l}$ is $a_k$-$a_l$-eulerian, $B_{m,n}$ is
$a_m$-$a_n$-eulerian, both intersect $Y$, and
$\{k,l,m,n\}=\{1,2,3,4\}$}

\medskip\noindent
Proof of Claim 2. Consider the graph $G/K$ obtained from $G$ by
contracting $K$ into a single vertex $b$. The contraction
transforms each edge $a_ib_i$ from $S$ into the edge $a_ib$
of~$G/K$. Since $G/K$ is $4$-edge-connected, by
Lemma~\ref{lemma:kconn-kfan} it contains four pairwise
edge-disjoint $b$-$Y$-paths, one through each edge $a_ib$. Take
the path containing the edge $a_ib$ and denote by $P_i$ its
segment starting at $a_i$; let $a_i'$ be the end-vertex of
$P_i$ in $B_Y\cap Y$. Thus $P_1$, $P_2$, $P_3$, and $P_4$ are
four pairwise edge-disjoint $A$-$Y$-paths entirely contained in
$B_Y$. Let $A'=\{a_1',a_2',a_3',a_4'\}$; again, some of these
vertices may coincide.

Consider the subgraph $B'=B_Y-\bigcup_i P_i$. Observe that each
odd-degree vertex of $B'$ lies in~$A'$. By
Lemma~\ref{lemma:oddvertices}, $B'$ contains an
$a_k'$-$a_l'$-path $P_{k,l}$ and an $a_m'$-$a_n'$-path
$P_{m,n}$ such that $P_{k,l}$ and $P_{m,n}$ are edge-disjoint
and $\{k,l,m,n\}=\{1,2,3,4\}$. If some vertex $a_k'\in A'$ has
even degree in $B'$, then there exists $l\ne k$ such that
$a_l'=a_k'$; in this case we can choose the path $P_{k,l}$ to
be trivial. Set $Q_{k,l}=P_k\cup P_{k,l}\cup P_l$ and
$Q_{m,n}=P_m\cup P_{m,n}\cup P_n$. Then $Q_{k,l}$ is an
$a_k$-$a_l$-eulerian subgraph and $Q_{m,n}$ is
$a_m$-$a_n$-eulerian subgraph of $B_Y$, and the subgraph
$B''=H-(Q_{k,l}\cup Q_{m,n})$ has all vertices of even degree.
Each component of $B''$ is eulerian and has at least one vertex
in either $Q_{k,l}$ or $Q_{m,n}$. We extend $Q_{k,l}$ and
$Q_{m,n}$ to subgraphs $B_{k,l}$ and $B_{m,n}$, respectively,
by attaching components of $B''$ to either $Q_{k,l}$ and
$Q_{m,n}$ in such a way that the resulting subgraphs $B_{k,l}$
and $B_{m,n}$ are connected. It is clear from the construction
that $B_{k,l}$ is an $a_k$-$a_l$-eulerian subgraph of $B_Y$ and
$B_{m,n}$ is an $a_m$-$a_n$-eulerian subgraph of $B_Y$. They
intersect $Y$ and form a decomposition of $B_Y$. This
establishes Claim~2.

\medskip\noindent

We continue with the proof of Case 2. Recall that $H$ is now
non-bipartite. It means that $H$ contains an odd circuit, and
hence an eulerian subgraph with an odd number of edges. Let $C$
be one with maximum number of edges. Let $B=B_C$ be the union
of all nontrivial components of $H-C$ that contain a vertex
from $A=\{a_1,a_2,a_3,a_4\}$ and let $D$ be the union of all
other nontrivial components. Note that each of $B$ and $D$ may
be empty, but since $H$ has an even number of edges, $B\cup D$
contains at least one component. Each component of $B\cup D$
has at least one vertex in~$C$ because $H$ is connected.

The maximality of $C$ implies that $B\cup D$ contains at most
one eulerian component, which necessarily has an odd number of
edges. Furthermore, $B$ has at most two components, which may
be either eulerian or semi-eulerian. We consider two subcases
according to whether $B$ has an even or an odd number of edges.

\medskip
\textbf{Subcase 2.1.} The subgraph $B$ has an even number of
edges. Since $C$ is odd, $D$ must be nonempty. We now apply
Claim~2 with $Y=D$. This implies that $B_Y=B\cup C$ and Claim~2
guarantees a decomposition of $B\cup C$ into semi-eulerian
subgraphs $B_{k,l}$ and $B_{m,n}$. Since $B\cup C$ is odd, one
of $B_{k,l}$ and $B_{m,n}$ is odd, say $B_{k,l}$. Furthermore,
$B_{m,n}$ intersects $D$, so $B_{m,n}\cup D$ is connected and
is also odd. Thus we can set $H_{k,l}=B_{k,l}$ and
$H_{m,n}=B_{m,n}\cup D$, which is the required decomposition of
$H$.

\medskip
\textbf{Subcase 2.2.} The subgraph $B$ has an odd number of
edges. Then $D$ must be empty. Let us take $Y=C$ thereby
obtaining $B_Y=B$. By Claim~2, we can decompose $B$ into two
semi-eulerian subgraphs $B_{k,l}$ and $B_{m,n}$ one of which is
odd, say $B_{k,l}$. As $B_{m,n}$ intersects~$C$, we can set
$H_{k,l}=B_{k,l}$ and $H_{m,n}=B_{m,n}\cup C$, which is the
required decomposition of $H$. This establishes Subcase~2.1 and
completes the proof of Proposition~\ref{prop:4-cuts}.
\end{proof}

\section{Concluding remarks}

As mentioned in Introduction, the existence of rooted $3$-odd
decompositions of eulerian graphs is closely related to the
existence of nowhere-zero integer $3$-flows in signed eulerian
graphs. Part (c) of Main Theorem in \cite{euler} states that a
signed eulerian graph admits a nowhere-zero integer $3$-flow
but not a nowhere-zero $2$-flow if and only if it admits a
rooted decomposition into three eulerian subgraphs with an odd
number of negative edges each. This result can be used to
derive a necessary and sufficient condition for an unsigned
eulerian graph to admit a rooted $3$-odd decomposition. For
this purpose, let us define an \textit{undirected nowhere-zero
integer $k$-flow} on a graph $G$ as a mapping $\phi\colon
E(G)\to\mathbb{Z}$ such that for each edge $e$ one has
$|\phi(e)|<k$ and $\phi(e)\ne 0$. The concept of an undirected
integer flow is easily seen to be equivalent to an integer flow
on the signed graph obtained by equipping $G$ with the
all-negative signature. Let us note that under the term
\textit{zero-sum flow} undirected integer flows were studied by
Akbari et al. in \cite{zero-sum1, zero-sum2}.

By using items (b) and (c) of Main Theorem of \cite{euler}
restricted to the all-negative signatures we obtain the
following result.

\begin{theorem}
An eulerian graph admits a rooted $3$-odd decomposition if and
only if it has an odd number of edges and admits an undirected
nowhere-zero integer $3$-flow.
\end{theorem}

\subsection*{Acknowledgmets}
Research reported in this paper was supported by the following
grants. The first author was partially supported by VEGA
1/0474/15 and the second author was partially supported by VEGA
1/0876/16. Both authors were also supported from APVV-0223-10
and APVV-15-0220.

\end{document}